\documentclass{amsart}
\usepackage{amsfonts, amsbsy, amsmath, amssymb}

\newtheorem{thm}{Theorem}[section]
\newtheorem{lem}[thm]{Lemma}

\newtheorem{fac}[thm]{Fact}

\newtheorem{thm-con}[thm]{Theorem-Conjecture}
\numberwithin{equation}{section}

\theoremstyle{definition}

\allowdisplaybreaks

\newcommand{\f}{\Bbb F}

\begin{document}

\title{On the DLW Conjectures}

\author[Xiang-dong Hou]{Xiang-dong Hou}
\address{Department of Mathematics and Statistics,
University of South Florida, Tampa, FL 33620}
\email{xhou@usf.edu}


\subjclass[2000]{}

\begin{abstract}
In 2007, Dmytrenko, Lazebnik and Williford posed two related conjectures about polynomials over finite fields. Conjecture~1 is a claim about the uniqueness of certain monomial graphs. Conjecture~2, which implies Conjecture~1, deals with certain permutation polynomials of finite fields. Two natural strengthenings of Conjecture~2, referred to as Conjectures~A and B in the present paper, were also insinuated. In a recent development, Conjecture~2 and hence Conjecture~1 have been confirmed. The present paper gives a proof of Conjecture~A.
\end{abstract}

\maketitle

\section{Introduction}\label{s1}

Let $\f_q$ denote the finite fields with $q$ elements. For $f,g\in\f_q[X,Y]$, $G_q(f,g)$ is a bipartite graph with vertex partitions $P=\f_q^3$ and $L=\f_q^3$, and edges defined as follows: a vertex $(p_1,p_2,p_3)\in P$ is adjacent to a vertex $[l_1,l_2,l_3]\in L$ if and only if
\begin{equation}\label{1.1}
p_2+l_2=f(p_1,l_1)\quad \text{and}\quad p_3+l_3=g(p_1,l_1).
\end{equation}
The graph $G_q(f,g)$ is called a {\em polynomial graph}, and when $f$ and $g$ are both monomials, it is called a {\em monomial graph}. Polynomial graphs were introduced by Lazebnik, Ustimenko and Woldar in \cite{Lazebnik-Ustimenko-Woldar-BAMS-1995} to provide examples of dense graphs of high girth. In particular, the monomial graph $G_q(XY,XY^2)$ has girth $8$, and its number of edges achieves the maximum asymptotic magnitude of the function $g_3(n)$ as $n\to \infty$, where $g_k(n)$ is the maximum number of edges in a graph of order $n$ and girth $\ge 2k+1$. (For surveys on the function $g_k(n)$, see \cite{Bondy-BSMS-2002, Furedi-Simonovits-BSMS-2013}.)

Let $q=p^e$, where $p$ is an odd prime and $e\ge 0$. It was proved in \cite{Dmytrenko-Lazebnik-Williford-FFA-2007} that every monomial graph of girth $\ge 8$ is isomorphic to $G_q(XY,X^kY^{2k})$ for some $1\le k\le q-1$, and the following conjecture was posed in \cite{Dmytrenko-Lazebnik-Williford-FFA-2007}:

\medskip

\noindent{\bf Conjecture~1.} (\cite[Conjecture~4]{Dmytrenko-Lazebnik-Williford-FFA-2007}) {\em 
Every monomial graph of girth $8$ is isomorphic to $G_q(XY,XY^2)$.
}
\medskip

To prove Conjecture~1, it suffices to show that if $1\le k\le q-1$ is such that $G_q(XY,X^kY^{2k})$ has girth $\ge 8$, then $k$ is a power of $p$. 

A polynomial $h\in\f_q[X]$ is called a {\em permutation polynomial} (PP) of $\f_q$ if the mapping $x\mapsto h(x)$ is a permutation of $\f_q$. For $1\le k\le q-1$, let
\begin{equation}\label{1.2}
A_k=X^k\bigl[(X+1)^k-X^k\bigr]\in\f_q[X]
\end{equation}
and 
\begin{equation}\label{1.3}
B_k=\bigl[(X+1)^{2k}-1\bigr]X^{q-1-k}-2X^{q-1}\in\f_q[X].
\end{equation}
It was proved in \cite{Dmytrenko-Lazebnik-Williford-FFA-2007} that if $1\le k\le q-1$ is such that $G_q(XY,X^kY^{2k})$ has girth $\ge 8$, then both $A_k$ and $B_k$ are PPs of $\f_q$. Consequently, a second conjecture was proposed:

\medskip
\noindent{\bf Conjecture~2.} (\cite[Conjecture~16]{Dmytrenko-Lazebnik-Williford-FFA-2007}) {\em If $1\le k\le q-1$ is such that both $A_k$ and $B_k$ are PPs of $\f_q$, then $k$ is a power of $p$.
}

\medskip

Note that if $k$ is a power of $p$, then $A_k$ and $B_k$ are clearly PPs of $\f_q$. Obviously, Conjecture~2 implies Conjecture~1. Although the polynomials $A_k$ and $B_k$ are both related to the graph $G_q(XY,X^kY^{2k})$, it is not clear how they are related to each other. Therefore, it is natural to consider the polynomials $A_k$ and $B_k$ separately, giving rise to the following two stronger versions of Conjecture~2; see \cite{Hou-Lappano-Lazebnik-FFA-2017,Kronenthal-FFA-2012,Lazebnik-Sun-Wang-LNSIM-2016}.

\medskip
\noindent{\bf Conjecture~A.} {\em 
Assume that $1\le k\le q-1$. Then $A_k$ is a PP of $\f_q$ if and only if $k$ is a power of $p$.
} 

\medskip
\noindent{\bf Conjecture~B.} {\em 
Assume that $1\le k\le q-1$. Then $B_k$ is a PP of $\f_q$ if and only if $k$ is a power of $p$.
} 

\medskip

We refer to all above conjectures as the {\em DLW conjectures} (after the authors of \cite{Dmytrenko-Lazebnik-Williford-FFA-2007}). A breakthrough on these conjectures came recently when Conjecture~2 and hence Conjecture~1 were proved in \cite{Hou-Lappano-Lazebnik-FFA-2017}. However, Conjectures~A and B remained unsolved. The purpose of the present paper is to give a proof of Conjecture~A.

We rely on several previous results on the polynomial $A_k$ from \cite{Hou-Lappano-Lazebnik-FFA-2017}. A summary of these results is given in Section~\ref{s2}. Section~\ref{s3} is devoted to the proof of Conjecture~A.

\section{Previous Results on $A_k$}\label{s2}

Recall that $q=p^e$, where $p$ is an odd prime and $e>0$, and $1\le k\le q-1$. Congruence of integers modulo $p$ is denoted by $\equiv_p$. For each integer $a>0$, let $a^*\in\{1,\dots,q-1\}$ be such that $a^*\equiv a\pmod{q-1}$; in addition, we define $0^*=0$. We will need the following known facts about the polynomial $A_k$ for the proof of Conjecture~A; the proofs of these facts can be found in \cite{Hou-Lappano-Lazebnik-FFA-2017}.

\begin{fac}\label{F2.1}
$A_k$ is a PP of $\f_q$ if and only if $\text{\rm gcd}(k,q-1)=1$ and
\begin{equation}\label{2.1}
\sum_{1\le i\le q-2}(-1)^i\binom si\binom{(ki)^*}{(2ks)^*}\equiv_p 0\quad\text{for all}\ 1\le s\le q-2.
\end{equation}
\end{fac}

\begin{fac}\label{F2.2}
Assume that $A_k$ is a PP of $\f_q$ and let $k'\in\{1,\dots,q-2\}$ be such that $k'k\equiv 1\pmod{q-1}$. Then all the base $p$ digits of $k'$ are $0$ or $1$.
\end{fac}

\begin{fac}\label{F2.3}
Conjecture~A is true for $q=p^e$, where $e=1$ or the greatest prime divisor of $e$ is $\le p-1$. In particular, Conjecture~A is true for $q=p^2$ with $e\le 2$.
\end{fac}

\section{Proof of Conjecture~A}\label{s3}

We first restate equation~\eqref{2.1} in terms of $k'$ defined in Fact~\ref{F2.2}.

\begin{lem}\label{L3.1}
$A_k$ is a PP of $\f_q$ if and only if $\text{\rm gcd}(k,q-1)=1$ and
\begin{equation}\label{3.1}
\sum_{2\le i\le q-2}(-1)^i\binom{(k's)^*}{(k'i)^*}\binom i{2s}\equiv_p 0,\quad 1\le s\le(q-1)/2,
\end{equation}
\begin{equation}\label{3.2}
\sum_{2\le i\le q-2}(-1)^i\binom{(k'(s+(q-1)/2))^*}{(k'i)^*}\binom i{2s}\equiv_p 0,\quad 1\le s<(q-1)/2.
\end{equation}
\end{lem}

\begin{proof} By Fact~\ref{F2.1}, we only have to show that \eqref{2.1} is equivalent the combination of \eqref{3.1} and \eqref{3.2}. Replacing $s$ by $(k's)^*$ and $i$ by $(k'i)^*$ in \eqref{2.1} gives
\begin{equation}\label{3.3}
\sum_{1\le i\le q-2}(-1)^i\binom{(k's)^*}{(k'i)^*}\binom i{(2s)^*}\equiv_p 0,\quad 1\le s\le q-2.
\end{equation}
Note that \eqref{3.1} is \eqref{3.3} with $1\le s\le (q-1)/2$ and that  \eqref{3.2} is \eqref{3.3} with $(q-1)/2<s<q-1$.
\end{proof}

For each integer $l\ge 0$, write $l^*=\sum_{i=0}^{e-1}l_ip^i$, $0\le l_i\le p-1$, and define
\[
d(l)=(l_0,\dots,l_{e-1})
\]
and 
\[
\text{supp}\,(l)=\{0\le i\le e-1:l_i>0\}.
\]
For $\alpha=(\alpha_0,\dots,\alpha_{e-1})$, $\beta=(\beta_0,\dots,\beta_{e-1})\in\Bbb Z^e$, the congruence $\alpha\equiv\beta\pmod{p-1}$ means that $\alpha_i\equiv\beta_i\pmod{p-1}$ for all $0\le i\le e-1$.

\begin{lem}\label{L3.2}
Assume that $e\ge 3$. Let $0<t\le e-1$, $0\le u,v\le(p-1)/2$, $2s=q-1-2(u+vp^t)$, and $l=\sum_{i=0}^{e-1}l_ip^i$, where $l_i\in\{0,1\}$, and $(l_0,\dots,l_{e-1})\ne(1,\dots,1)$. Then 
\begin{equation}\label{3.4} 
\begin{split}
&\sum_{2\le i\le q-2}(-1)^i\binom{(l(s+(q-1)/2))^*}{(li)^*}\binom i{2s}\cr
\equiv_p\,&\sum_{\substack{0\le a\le 2u\cr 0\le b\le 2v}}(-1)^{(a+b+u+v)(x+y)}{\binom au}^x{\binom bv}^x{\binom{a+b}{u+v}}^y\binom{2u}a\binom{2v}b,
\end{split}
\end{equation}
where 
\begin{align*}
y=\,&|\text{\rm supp}\,(l)\cap\text{\rm supp}\,(p^tl)|=\bigl|\{0\le i\le e-1:l_i=l_{i-t}=1\}\bigr|,\cr
x=\,&|\text{\rm supp}\,(l)\setminus\text{\rm supp}\,(p^tl)|=\sum_{i=0}^{e-1}l_i-y,
\end{align*}
and the subscript of $l_{i-t}$ is taken modulo $e$.
\end{lem}

\begin{proof}
$1^\circ$ We have
\begin{equation}\label{3.5}
d(2s)=(\,\underset{0}{p-1-2u},\, p-1,\dots,p-1,\,\underset{t}{p-1-2v}, p-1,\dots,p-1\,).
\end{equation}
Hence, if $2\le i\le q-2$ is such that $\binom i{2s}\not\equiv_p 0$, we must have 
\begin{equation}\label{3.6}
d(i)=(\,\underset{0}{p-1-a},\, p-1,\dots,p-1,\,\underset{t}{p-1-b}, p-1,\dots,p-1\,),
\end{equation}
where $0\le a\le 2u$ and $0\le b\le 2v$; in this case,
\begin{equation}\label{3.7}
\begin{split}
\binom i{2s}\,&\equiv_p\binom{p-1-a}{p-1-2u}\binom{p-1-b}{p-1-2v}\equiv_p\binom{-(a+1)}{2u-a}\binom{-(b+1)}{2v-b}\cr
&=(-1)^{a+b}\binom{2u}a\binom{2v}b.
\end{split}
\end{equation}

\medskip

$2^\circ$ We claim that when \eqref{3.6} is satisfied, 
\begin{equation}\label{3.8}
\binom{(l(s+(q-1)/2))^*}{(li)^*}\equiv_p(-1)^{(a+b+u+v)(x+y)}{\binom au}^x{\binom bv}^x{\binom{a+b}{u+v}}^y.
\end{equation}

For the sake of notational convenience, we write
\begin{align}
\label{3.8.0}
d(l)=\,&(\,1, \dots, 1,\, 1, \dots, 1,\, 0, \dots, 0,\, 0, \dots, 0\,),\\
\label{3.9}
d(p^{t}l)=\,&(\,\underbrace{0, \dots, 0}_x,\, \underbrace{1, \dots, 1}_y,\, \underbrace{1, \dots, 1}_x,\, 0, \dots, 0\,);
\end{align}
In doing so, we no longer maintain the alignment of the components in \eqref{3.8.0} and \eqref{3.9} with the powers of $p$. Since $l(s+(q-1)/2)\equiv -(u+vp^{t})l\pmod{q-1}$, we have 
\begin{equation}\label{3.10}
\begin{split}
&d(l(s+(q-1)/2))=\cr
&(\,\underbrace{p-1-u,\dots,p-1-u}_x,\, \underbrace{p-1-(u+v),\dots,p-1-(u+v)}_y,\cr
&\,\,\underbrace{p-1-v,\dots,p-1-v}_x,\, p-1,\dots, p-1\,).
\end{split}
\end{equation}

We first assume that $a+b\le p-1$. Since $li\equiv-(a+bp^{t})\pmod{q-1}$, we have 
\begin{equation}\label{3.11}
\begin{split}
&d(li)=\cr
&(\,\underbrace{p-1-a,\dots,p-1-a}_x,\, \underbrace{p-1-(a+b),\dots,p-1-(a+b)}_y,\cr
&\,\,\underbrace{p-1-b,\dots,p-1-b}_x,\, p-1,\dots, p-1\,).
\end{split}
\end{equation}
Therefore
\begin{equation}\label{3.12}
\binom{(l(s+(q-1)/2))^*}{(li)^*}\equiv_p{\binom{p-1-u}{p-1-a}}^x{\binom{p-1-v}{p-1-b}}^x{\binom{p-1-(u+v)}{p-1-(a+b)}}^y.
\end{equation} 
In the above,
\begin{align}
\label{3.13}
\binom{p-1-u}{p-1-a}\,&\equiv_p (-1)^{a+u}\binom au,\\
\label{3.14}
\binom{p-1-v}{p-1-b}\,&\equiv_p (-1)^{b+v}\binom bv,\\
\label{3.15}
\binom{p-1-(u+v)}{p-1-(a+b)}\,&\equiv_p (-1)^{a+b+u+v}\binom {a+b}{u+v}.
\end{align}
Hence
\begin{align*}
\binom{(l(s+(q-1)/2))^*}{(li)^*}&\equiv_p(-1)^{(a+u)x}{\binom au}^x(-1)^{(b+v)x}{\binom bv}^x (-1)^{(a+b+u+v)y}{\binom {a+b}{u+v}}^{\!y}\cr
&=(-1)^{(a+b+u+v)(x+y)}{\binom au}^x{\binom bv}^x{\binom{a+b}{u+v}}^y.
\end{align*}

If $a+b>p-1$ but $y=0$, the above computation also gives \eqref{3.8}.

\medskip

Now assume that $a+b>p-1$ and $y>0$. Then 
\[
\binom{a+b}{u+v}\equiv_p\binom{a+b-p}{u+v}=0
\]
since $0\le a+b-p\le 2(u+v)-p<u+v$. It remains to show that the left side of \eqref{3.8} is also $\equiv_p 0$. We have
\begin{equation}\label{3.16}
\begin{split}
&d(li)\equiv\cr
&(\,\underbrace{p-1-a,\dots,p-1-a}_x,\, \underbrace{p-1-(a+b),\dots,p-1-(a+b)}_y,\cr
&\,\,\underbrace{p-1-b,\dots,p-1-b}_x,\, p-1,\dots, p-1\,) \pmod{p-1}.
\end{split}
\end{equation}
We remind the reader that for notational convenience, the components of the right side of \eqref{3.16} have not been aligned with the powers of $p$. Now, however, it necessary to align these components with the powers of $p$ since carries in base $p$ will be considered. 

\medskip

{\bf Case 1.} Assume that at least one component of the right side of \eqref{3.16} is $>0$. Since $y>0$, we may write
\[
d(li)\equiv (\alpha_1,\dots,\alpha_m)\pmod{p-1},
\]
where each $\alpha_j$ ($1\le j\le m$) is a block of the form
\[
\alpha_j=(p-1-(a+b),\,\dagger,\dots,\dagger,\,\epsilon_j,\,*,\dots,*),
\]
where each $\dagger$ is either $p-1-(a+b)$ or $0$, each $*$ belongs $\{0,\dots,p-1\}$, and $0<\epsilon_j\le p-1$. It follows that
\begin{equation}\label{3.17}
d(li)=(\alpha_1',\dots,\alpha_m'),
\end{equation}
where
\[
\alpha_j'=(2p-1-(a+b),\,*,\dots,*,\,\epsilon_j-1,\, *,\dots,*).
\]
Align the components of \eqref{3.10} with those of \eqref{3.17}. This gives
\[
d(l(s+(q-1)/2))=(p-1-(u+v),*,\dots,*).
\]
Therefore
\[
\binom{(l(s+(q-1)/2))^*}{(li)^*}\equiv_p\binom{p-1-(u+v)}{2p-1-(a+b)}\cdots=0
\]
since $2p-1-(a+b)\ge 2p-1-2(u+v)>p-1-(u+v)$.

\medskip

{\bf Case 2.} Assume that every component of the right side of \eqref{3.16} is either $p-1-(a+b)$ or $0$. Since $(l_0,\dots,l_{e-1})\ne(1,\dots,1)$, we have $y<e$ and hence $x>0$. It follows from \eqref{3.16} that $a=b=p-1$. Therefore we may write
\[
d(li)\equiv(0,\dagger,\dots,\dagger)\pmod{p-1},
\]
where each $\dagger$ is either $0$ or $p-(a+b)$. Hence
\begin{equation}\label{d(li)}
d(li)=(p-1,*,\dots,*).
\end{equation}
Align the components of \eqref{3.10} with those of \eqref{d(li)}. Without loss of generality, we may write
\[
d(l(s+(q-1)/2))=(p-1-u,*,\dots,*).
\]
Since $a/2\le u\le (p-1)/2$, we have $u=(p-1)/2$. Therefore,
\[
\binom{(l(s+(q-1)/2))^*}{(li)^*}\equiv_p\binom{p-1-u}{p-1}\cdots=0.
\]

\medskip

$3^\circ$ Equation~\eqref{3.4} follows from \eqref{3.7} and \eqref{3.8}. (Note that $(-1)^i=(-1)^{a+b}$ by \eqref{3.6}.)
\end{proof}

\begin{proof}[Proof of Conjecture~A]
Assume that $A_k$ is a PP of $\f_q$. We show that $k$ is a power of $p$, equivalently, $k'$ is a power of $p$. By Fact~\ref{F2.3}, we may assume that $e\ge 3$. By Fact~\ref{F2.2}, $k'=\sum_{i=0}^{e-1}l_ip^i$, where $l_i\in\{0,1\}$. Assume to the contrary that $\sum_{i=0}^{e-1}l_i>1$. Since $\text{gcd}(k',q-1)=1$, we have $(l_0,\dots,l_{e-1})\ne(1,\dots,1)$. We follow the notation of Lemma~\ref{L3.2}. In Lemma~\ref{L3.2}, let $u=v=(p-1)/2$ and $l=k'$. Moreover, choose $0< t\le e-1$ such that $y>0$; this is possible since $l_i=1$ for at least two $i$. Also note that $s=(q-1)/2-(u+vp^{t})<(q-1)/2$. Now we have 
\begin{equation}\label{3.19}
\begin{split}
0\,&\equiv_p \sum_{2\le i\le q-2}(-1)^i\binom{(k'(s+(q-1)/2))^*}{(k'i)^*}\binom i{2s}\kern 3.65cm\text{(by \eqref{3.2})}\cr
&\equiv_p \sum_{\substack{0\le a\le 2u\cr 0\le b\le 2v}}(-1)^{(a+b+u+v)(x+y)}{\binom au}^x{\binom bv}^x{\binom{a+b}{u+v}}^y\binom{2u}a\binom{2v}b \kern 1.02cm\text{(by \eqref{3.4})} \cr
&= \sum_{(p-1)/2\le a,\,b\le p-1}(-1)^{a+b}{\binom a{(p-1)/2}}^x{\binom b{(p-1)/2}}^x{\binom {a+b}{p-1}}^y\binom{p-1}a\binom{p-1}b.
\end{split}
\end{equation}
(For that last step of \eqref{3.19}, note that $x+y$ is odd since $\text{gcd}(k',q-1)=1$.) For $(p-1)/2\le a,b\le p-1$, $\binom{a+b}{p-1}\not\equiv_p 0$ only if $a=b=(p-1)/2$. Hence 
\[
\text{the right side of \eqref{3.19}}\equiv_p\binom{p-1}{(p-1)/2}^2=1\not\equiv_p 0,
\]
which is a contradiction.
\end{proof}

\noindent{\bf Remark.} The above proof of Conjecture~A uses Lemma~\ref{L3.2} only for $u=v=(p-1)/2$. We choose to present Lemma~\ref{3.2} in a more general setting in anticipation of possible applications of the result in related problems.



\end{document}